\documentclass[a4paper,11pt,oneside,reqno]{amsart}
\usepackage{amsfonts,amsmath, amssymb,amsthm,amscd}
\usepackage[a4paper,mpexclude,DIV13]{typearea}
\usepackage{verbatim}
\usepackage{graphicx}
\usepackage[latin1]{inputenc}
\newtheorem{theorem}{Theorem}
\newtheorem{corollary}{Corollary}
\newtheorem{lem}[theorem]{Lemma}
\newtheorem{prop}[theorem]{Proposition}

\theoremstyle{remark}
\newtheorem{rem}{Remark}

\usepackage{mathrsfs}
\usepackage{verbatim}
\newcommand{\N}{\ensuremath{\mathbb{N}}}
\newcommand{\R}{\ensuremath{\mathbb{R}}}

\newcommand{\E}{\ensuremath{\mathbb{E}}}
\newcommand{\Pro}{\ensuremath{\mathbb{P}}}
\newcommand{\indi}{\ensuremath{\textbf{1}}}

\def\<{\langle}
\def\>{\rangle}

\begin{document}
\title{Poisson convergence for the largest eigenvalues of Heavy Tailed Random Matrices}

\author[A. Auffinger]{Antonio Auffinger}
\address{A. Auffinger\\
  Courant Institute of the Mathematical Sciences\\
  New York University\\
  251 Mercer Street\\ 
  New York, NY 10012, USA}
\email{auffing@cims.nyu.edu}

\author[G. Ben Arous]{G\'erard Ben Arous}
\address{G. Ben Arous\\
  Courant Institute of the Mathematical Sciences\\
  New York University\\
  251 Mercer Street\\ 
  New York, NY 10012, USA}
\email{benarous@cims.nyu.edu}

\author[S. P\'{e}ch\'{e}]{Sandrine P\'{e}ch\'{e}}
\address{S. P\'{e}ch\'{e}\\ 
  Institut Fourier\\
  Universit\'{e} Joseph Fourier - Grenoble \\
  100 rue des Maths, BP 74\\
  38402 St Martin d'Heres, France}
\email{sandrine.peche@ujf-grenoble.fr}


\date{\today}

\maketitle


\begin{abstract}
On \'{e}tudie la loi des plus grandes valeurs propres de matrices 
al\'{e}atoires sym\'{e}triques r\'{e}elles et de covariance empirique quand les coefficients des matrices sont \`{a} queue lourde. On \'{e}tend le r\'{e}sultat obtenu par A. Soshnikov dans \cite{Sos1} et on montre que le comportement asymptotique des plus grandes valeurs propres est d\'{e}termin\'{e} par les plus grandes entr\'{e}es de la matrice.

\end{abstract}

\begin{abstract}
We study the statistics of the largest eigenvalues of real symmetric and sample covariance matrices when the entries are heavy tailed. Extending the result obtained by A. Soshnikov in \cite{Sos1}, we prove that, in the absence of the fourth moment, the asymptotic behavior of the top eigenvalues is determined by the behavior of the largest entries of the matrix.
\end{abstract}

\section{Introduction and Notation}

We study the statistics of the largest eigenvalues of symmetric and sample covariance matrices when the entries are heavy tailed. Extending the result obtained by Soshnikov in \cite{Sos1}, we prove that in the absence of a finite fourth moment, the asymptotic behavior of the top eigenvalues is determined by the behavior of the largest entries of the matrix, i.e that the point process of the largest eigenvalues (properly normalized) converges to a Poisson Point Process, as in the usual extreme value theory for i.i.d. random variables. This result was predicted in the physics literature by Biroli, Bouchaud and Potters \cite{Bou}.   

We first consider the case of random real symmetric matrices with independent and heavy tailed entries. Let $(a_{ij})$, $1 \leq i \leq n, 1 \leq j \leq n$ be i.i.d random variables such that:

\begin{equation}
\label{def1}
1-F(x)=\bar{F}(x)=\Pro(|a_{ij}|>x)=L(x)x^{-\alpha},
\end{equation}  
where $\alpha > 0$ and $L$ is a slowly varying function, i.e., for all $t>0$ $$\lim_{x\rightarrow \infty} \frac{L(tx)}{L(x)}=1.$$


Consider the $n \times n$ real symmetric random matrix $A_n$ whose entries above the diagonal are the $(a_{ij}) , 1 \leq i \leq j \leq n$. Hypothesis (\ref{def1}) would be natural in the theory of extreme values for i.i.d random variables. It simply asserts that the distribution of the entries is in the max-domain of attraction of the Frechet distribution with exponent $\alpha$ (see \cite{Res}, page 54). Thus, for any $\alpha > 0$ the point process of extreme values of the entries of $A_n$ (properly normalized) is asymptotically Poissonian. More precisely, let 

\begin{equation}\label{bnw}
b_{n}=\inf\{x: 1 - F(x) \leq \frac{2}{n(n+1)}\},
\end{equation}
then the Point Process $$\hat{\mathcal{P}_n} = \sum_{1\leq i \leq j \leq n} \delta_{b_n^{-1}|a_{ij}|}$$ converges to a Poisson Point Process with intensity 

$$\rho(x)= \frac{\alpha}{ x^{1 + \alpha}}.$$ It is also classical that there exists another slowly varying function $L_o$ such that 
\begin{equation}
b_n \sim L_o(n)n^{\frac{2}{\alpha}}.
\end{equation} When $L \equiv 1$, then $b_n = (\frac{n(n+1)}{2})^{1/\alpha}$. 

We denote by $\lambda_1 \geq \ldots \geq \lambda_{n}$ the $n$ (real) eigenvalues of $A_n$ and we consider the point process on $(0, \infty)$ of (normalized) positive eigenvalues of $A_n$:

$$\mathcal{P}_n = \sum \delta_{b_n^{-1}\lambda_i}\indi_{\lambda_i > 0}.$$
\begin{theorem}\label{thmw}
We assume (\ref{def1}) with $0 < \alpha < 4$. For $2 \leq \alpha < 4$ we also assume that the entries are centered, i.e $\E (a_{ij})=0$. The random point process $\mathcal{P}_n$ converges in distribution  to the Poisson Point Process $\mathcal{P}$ defined on $(0, \infty)$ with intensity $\rho(x)= \frac{\alpha}{x^{1 + \alpha}}.$
\end{theorem}

This result thus shows that the largest eigenvalues of $A_n$ behave as the largest entries of the matrix $A_n$ when $0 < \alpha < 4$. It was proved in the range $0 < \alpha < 2$ by Soshnikov \cite{Sos1}. It implies for instance that the maximum eigenvalue has a Fr\'{e}chet limit distribution:

\begin{corollary}\label{cor1}
\begin{equation}
\lim_{n\rightarrow\infty} \Pro (\frac{1}{b_{n}}\lambda_1 \leq x ) = \exp (-x^{-\alpha}).
\end{equation}
\end{corollary}

One word of comment is in order here. When the entries have light tails, it is well-known that the random field of largest eigenvalues is not Poissonian but determinantal, and that the fluctuations of the top eigenvalue are asymptotically distributed as in the GOE, i.e have a Tracy-Widom distribution \cite{Sos3}. We actually believe that the universal Tracy-Widom picture holds as soon as $\alpha>4$, see \cite{Bou} for a discussion and simulation. Some steps in this direction have been achieved by A. Ruzmaikina \cite{Ruz}, who proves that the Tracy-Widom limit holds for $\alpha$ large enough. (She claimed that $\alpha>18$ is enough, we believe that the arguments of \cite{Ruz} only work if $ \alpha > 36$, see Remark \ref{Ruzma}.)



Let us first consider the case $\alpha > 2$. It is well known that, if \begin{equation}\mu_n = \frac{1}{n}\sum \delta_{\frac{\lambda_i}{\sqrt{n}}} \end{equation} denotes the spectral measure of $\frac{A_n}{\sqrt{n}}$, $\mu_n$ converges weakly almost surely to a non-random limit, the semi-circle law,
\begin{equation}
\nu(x)= 
\begin{cases}
\frac{1}{2\pi\sigma^2}\sqrt{4\sigma^2 - x^2}, \quad \text{if} \quad |x|\leq 2\sigma \\
0, \quad \text{otherwise.}
\end{cases}
\end{equation} which depends only on the variance $\sigma^2$ of the entries. In the case where $\alpha > 4$, Bai and Yin \cite{Yin} have proved that the top eigenvalue sticks to the bulk, i.e that for all $\epsilon > 0$, $$\Pro (|\frac{1}{\sqrt{n}}\lambda_1 - 2\sigma| \leq \epsilon )\rightarrow 1 \quad, \quad n \rightarrow \infty.$$ This shows that for $\alpha > 4$, $b_n^{-1}\lambda_1 \rightarrow \infty$ so that our result in Corollary \ref{cor1} ceases to be true. Our result shows that $\frac{1}{\sqrt{n}}\lambda_1$ is, roughly speaking, of order $n^{2/\alpha - 1/2}$ and thus diverges. This is in agreement with and sharpens Bai and Yin's result, who have shown that a finite fourth moment is necessary to have the convergence to the edge of the bulk. 

The case $\alpha = 4$ with infinite fourth moment seems very interesting and still open. It might exhibit an interesting transition between the Tracy-Widom regime and the Poissonian one.  

Coming back now to the case $\alpha < 2$, the situation is different. The bulk itself is not a semi-circle. It was recently proved by Ben Arous-Guionnet \cite{GBA}  that the spectral measure $$\hat{\mu_n} = n^{-1}\sum \delta_{\frac{\lambda_i}{c_n}},$$ where $$c_n = \inf\{x: 1 - F(x) \leq \frac{1}{n}\}$$ converges to a limiting distribution $\mu_\alpha$. This limit probability distribution $\mu_{\alpha}$ is not compactly supported and has a polynomial tail of type $\frac{C_\alpha}{x^{1+\alpha}}dx$ for some constant $C_{\alpha}$. In this case ($\alpha < 2$), this is perfectly compatible with the present result: the extreme values of iid random variables with that distribution $\mu_{\alpha}$ would have exactly the behavior we have given.

We also study in this paper the behavior of the top of the spectrum for another very important family of random matrices, i.e the ensemble of large random sample covariance matrices.

In this setting, we consider $A_n$ a $n \times p$ random matrix with i.i.d centered entries $(a_{ij}), 1 \leq i \leq n, 1 \leq j \leq p$ and  we define as usual the sample covariance matrix $X_n = \frac{1}{p} A_nA_n^t$. The asymptotic behavior of the bulk of $X_n$ is also well-known by the classical result of Marchenko-Pastur \cite{Mar}, in the case where we assume a finite second moment. The case where $0< \alpha < 2$ is treated in \cite{GBA2}.

Similarly to (\ref{bnw}), let $$b_{np}=\inf\{x: 1 - F(x) \leq \frac{1}{np}\}.$$ If $\lambda_{1} \geq \ldots \geq \lambda_n$ are the ordered eigenvalues of $A_nA_n^t$ and $\lim_{n\rightarrow\infty} \frac{p}{n} = \gamma$ for some positive constant $\gamma \geq 1$ defining $$\mathcal{P}_n =\sum_i \delta_{b_{np}^{-2}\lambda_i}$$  we have the following:

\begin{theorem}\label{thm1}
We assume (\ref{def1}) with $0 < \alpha < 4$. For $2 \leq \alpha < 4$, we also assume that the entries are centered, i.e, $\E (a_{ij})=0$. The random point process $\mathcal{P}_n$ converges in distribution, as $p$ goes to infinity, to the Poisson Point Process $\mathcal{P}$ defined on $(0, \infty)$ with intensity $\rho(x)= \frac{\alpha}{2x^{1 + \alpha/2}}.$

\end{theorem} 

Again, as a simple corollary, we obtain the behavior of the maximal eigenvalue: 

\begin{corollary}\label{cor2}
\begin{equation}
\lim_{n\rightarrow\infty} \Pro (\frac{1}{b_{np}^2}\lambda_1 \leq x ) = \exp (-x^{-\frac{\alpha}{2}}).
\end{equation}
\end{corollary}

The rest of the paper is organized as follows. First, in section \ref{Soss} we recall briefly the main results contained in \cite{Sos1}, i.e, the proof of Theorem \ref{thmw} in the case $0 < \alpha < 2$. We then prove Theorem \ref{thm1} in the case $0 < \alpha < 2$ in Section \ref{cov0}. We then study in section \ref{wigner} and \ref{cov4} the case where $2 \leq \alpha < 4$ for the Wigner and the sample covariance matrix cases respectively. This is a bit different in nature since we now have to perform a more subtle separation of scales. This is done through an estimate of traces of high powers of our random matrices properly truncated. This combinatorial part of the proof draws on the work of Soshnikov \cite{Sos3} and on the recent work by P\'{e}ch\'{e}-Soshnikov \cite{SosPec}.


\section{Wigner matrices when $0 < \alpha < 2$.}\label{Soss}

In this section,  we will recall the results in the paper of Soshnikov \cite{Sos1} (see also \cite{Sos4} for precise statement of Lemma \ref{entries}).

Let $A_n$ be $n \times n$ random (real) symmetric matrix with iid entries satisfying (\ref{def1}) with $0 < \alpha < 2$. Also let $\lambda_1 \geq \lambda_2 \geq \lambda_3 \geq \ldots \geq \lambda_n$ be its eigenvalues and $a_{i_lj_l}$ its $l$-th largest entry in absolute value.   



In order to prove Theorem \ref{thmw} in this case, Soshnikov \cite{Sos1} proceeds as follows. The basic idea is to show that for each finite $k$, and for each given $\epsilon > 0$, 
\begin{equation}\label{s2}
\Pro(|\frac{\lambda_k}{a_{i_kj_k}} - 1| > \epsilon) \rightarrow 0 \quad \text{as} \quad n \rightarrow \infty.
\end{equation}

We first consider the case where $k=1$, which implies corollary \ref{cor1}. The following  crucial lemma is purely probabilistic. It describes how the largest entries are placed in the matrix. This lemma will be adapted in all other sections.

\begin{lem}\label{entries}
\begin{enumerate}

\item With probability going to one, there are no diagonal entries greater in absolute value than $b_n^{11/20}.$ 

\item With probability going to one, there is no pair $(i, j)$ such that $|a_{ij}| > b_n^{99/100}$ and $|a_{ii}| + |a_{jj}| > b_n^{1/10}.$

\item For any positive $\delta > 0$ with probability going to one there is no row that has at least two entries greater in absolute value than $b_n^{3/4}+\delta$.

\item With probability going to one, there is no row such that its maximum and the sum of the absolute value of the remaining elements in the row are both greater than $b_n^{3/4 + \alpha/8}.$
\end{enumerate}
\end{lem}

Once one has proved the previous lemma, the next step is to relate the entries of the matrix with its maximum eigenvalue $\lambda_1$. This can be done in two steps. 

First, one can bound from below the top eigenvalue using the Rayleigh-Ritz representation of $\lambda_1$:

\begin{equation}\label{minmax}
\lambda_1 = \sup_{v:|v|=1} \left\langle A_nv, v\right\rangle .
\end{equation}

Considering a well-chosen vector $v$ in terms of the position of the largest entry of $A_n$, (\ref{minmax}) will provide the inequality
$$\lambda_1 \geq a_{i_1j_1}(1 + o(1)).$$

Secondly, one studies the following norm of the matrix $A_n$, that is the norm of $A_n$ as a linear operator from $l_\infty$ to $l_\infty$:

\begin{equation}\label{norm}
\|A_n\|_{\infty} \equiv \max_{i} \sum_{j=1}^{n} |a_{ij}|.
\end{equation}

Since $A_n$ is symmetric, we can show that $\|A_n\|_{\infty}$ is an upper bound for $\lambda_1$.  Lemma \ref{entries} then relates $\|A_n\|_{\infty}$ to the maximum entry of the matrix $A_n$. In fact, Soshnikov showed that, given $\epsilon > 0$, there exists $\theta>0$ such that for $n$ sufficiently large, one has:

$$\Pro\left(\Big |\frac{\|A_n\|_{\infty}}{\max_{ij} |a_{ij}|} - 1\Big | > \epsilon\right) \leq n\exp(-n^{\theta}).$$ 

This proves (\ref{s2}) when $k = 1$.

Now, for any finite $k$, using Lemma \ref{entries}, it is possible to find some well-chosen unit vectors $v_k$, such that,
$$A_nv_k = a_{i_{k}j_{k}}v_k + r_k,\quad \text{and} \quad \|r_k\| = o(1).$$
This fact together with a standard result in pertubation theory of symmetric matrices (see for instance \cite{Bha} page 77) imply that $A_n$ has eigenvalues  $a_{i_{l}j_{l}} (1 + o(1)), 1 \leq l \leq k$, for any finite $k$.

To finally get Theorem \ref{thmw} we use induction on $k$ in (\ref{s2}), supported by Corollary \ref{cor1} and the following very classical result about symmetric matrices, that can be found for instance in \cite{Bha}, page 59.

\begin{prop}[Cauchy Interlacing Theorem]
Let $A_n$ be an $n \times n$ Symmetric matrix and $\lambda_1 \geq \lambda_2 \geq \ldots \geq \lambda_n$ its ordered (real) eigenvalues. If one considers the restriction $B$ of $A_n$ to any subspace of co-dimension $1$ and denotes by $\mu_1 \geq \mu_2 \geq \ldots \geq \mu_{n-1}$ the eigenvalues of $B$ then 
$$\lambda_1 \geq \mu_1 \geq \lambda_2 \geq \mu_2 \geq \ldots \geq \mu_{n-1} \geq \lambda_n.$$
\end{prop}

We briefly explain how to use the last statement in this setting. For simplicity, instead of describing the general step of the induction, we describe only the step $k=2$. First, one considers the submatrix obtained by removing the $i_1-$th row and the $j_1-$th column from $A_n$. Clearly, this submatrix has $a_{i_{2}j_{2}}$ as largest entry in absolute value, and by the interlacing property its largest eigenvalue will be greater than $\lambda_2$. Thus, one can apply Corollary \ref{cor1} to finally get (\ref{s2}) for $k=2$. 
\\

To get Theorem $\ref{thmw}$ as stated, i.e, to prove tightness and the convergence of $\mathcal{P}_n$ to $\mathcal{P}$, it suffices to prove, see \cite{Kal} Theorem $16.16$, that for all intervals $(a,b)$, where $0 < a < b$ one has that the random variable $\mathcal{P}_n(a,b)$ converges in distribution to $\mathcal{P}(a,b)$. This can be verified as follows. 

Since $$ \Pro (\mathcal{P}_n(a,\infty)>k)=\Pro(\frac{\lambda_{k+1}}{b_n} > a),$$ it is easy to see that $(\ref{s2})$  proves the convergence in distribution of $\mathcal{P}_n$ when restricted to an interval $(a,\infty)$, $a>0$, i.e the random variable $\mathcal{P}_n(a,\infty)$ converges in distribution to  $\mathcal{P}(a,\infty)$ which is a Poisson Process with parameter $a^{-\alpha}$.

To derive the result for a general interval, we first note that since $$\Pro(\mathcal{P}(\{b\})>0)=0,$$ we can consider intervals $(a,b]$ which can be written as the difference of $I_a=(a, \infty)$ and $I_b=(b, \infty)$. Now, 
\begin{equation}\label{tight}
\Pro(\mathcal{P}_n(a,b] = l) = \sum_{k=l}^{\infty} \Pro(\mathcal{P}_n(a, \infty)=k, \mathcal{P}_n(b, \infty)=k-l).
\end{equation}

Each term inside the sum converges to $\Pro(\mathcal{P}(a, \infty)=k, \mathcal{P}(b, \infty)=k-l)$ and is also bounded by $\Pro(\mathcal{P}_n(a,\infty)=k)$. Since the sum $$\sum_{k=l}^\infty \Pro(\mathcal{P}_n(a, \infty)=k)$$ is finite, an application of Fatou's Lemma shows that (\ref{tight}) converges to $$\Pro(\mathcal{P}(a,b] = l) = \sum_{k=l}^{\infty} \Pro(\mathcal{P}(a, \infty)=k, \mathcal{P}(b, \infty)=k-l).$$ 

The details omitted here will be considered in the next sections, specially the next one, where we follow the same strategy in the case of Sample Covariance matrices with $0 < \alpha < 2$. 


\section{Sample Covariance matrices when $0 < \alpha < 2$.}\label{cov0}

The proof of Theorem \ref{thm1} in this case is based on the following lemma, which is almost identical to Lemma \ref{entries} given in last section. 

\begin{lem}\label{lem1}

Let $A_n$ be a $n\times p$ random matrix with i.i.d. entries $a_{ij}$, $1 \leq i \leq n$, $1 \leq j \leq p$ satisfying (\ref{def1}). Also assume that $\lim_{n\rightarrow\infty} \frac{p}{n} = \gamma$ for some constant $\gamma \geq 1$.
Then:

\begin{enumerate}
	\item If $B_{np}^\delta$ is the event 'There is a row with $2$ entries greater than $b_{np}^\delta$ in absolute value' then 
	$$ \forall \delta > 3/4, \quad \lim_{p\rightarrow\infty} \Pro(B_{np}^\delta) = 0.$$
	
	\item Also, $$\lim_{n\rightarrow\infty} \Pro(\exists i,  1 \leq i \leq n, \quad \max_{1 \leq j \leq p} |a_{ij}| > b_{np}^{\frac{3}{4}+\frac{\alpha}{8}} \quad \text{and} \quad \sum_{j=1}^{p} |a_{ij}| - \max_{1 \leq j \leq p} |a_{ij}| >b_{np}^{\frac{3}{4}+\frac{\alpha}{8}})=0.$$
	
	\item Similarly, $$\lim_{n\rightarrow\infty} \Pro(\exists j,  1 \leq j \leq p, \quad \max_{1 \leq i \leq n} |a_{ij}| >b_{np}^{\frac{3}{4}+\frac{\alpha}{8}} \quad \text{and} \quad \sum_{j=1}^{n} |a_{ij}| - \max_{1 \leq j \leq n} |a_{ij}| >b_{np}^{\frac{3}{4}+\frac{\alpha}{8}})=0.$$
\end{enumerate}
\end{lem}
\subsection{Proof of Lemma \ref{lem1}}
\begin{proof}[Proof of (a)]

First, a basic fact from slowly varying functions that will be used repeatedly in this paper (see for instance \cite{Bih}, chapter I). Given any $\delta>0$, one has that  
\begin{equation}\label{Luc}
x^{-\delta} \ll L(x) \ll x^{\delta} \quad \text{as} \quad x \rightarrow \infty.   
\end{equation}
We recall that $f(x) \ll g(x)$ means that the ratio $\frac{f(x)}{g(x)}$ tends to $0$ as $x$ tends to infinity.
Hence, using (\ref{Luc}) and choosing $\delta = \frac{3}{4} + \epsilon$, 
\begin{eqnarray*}
\Pro (B_{np}^\delta) = \Pro (\exists i \leq n, \exists j, k, j \neq k, s.t. |a_{ij}| \geq b_{np}^{3/4 + \epsilon} \text{and} |a_{ik}| \geq b_{np}^{3/4 + \epsilon}) & \leq & p^2n(1-F(b_{np}^{3/4 + \epsilon}))^2 \\
&\leq& p^2nL(b_{np}^{3/4 + \epsilon})^2\frac{1}{b_{np}^{3\alpha/2 + 2\epsilon\alpha}} \\
&=& o(n^{-4\epsilon + \theta})  
\end{eqnarray*}
for a small enough $\theta$, since for any $\epsilon > 0$, $b_{np}^{3\alpha/2 + 2\epsilon\alpha} \gg n^3$ by the definition of $b_{np}$.

\vspace{0.5cm}
\emph{Proof of (b).}
We split the proof in two cases. The idea in both cases is the same and the computation almost identical.

We start by assuming that $1 < \alpha < 2$.

\vspace{0.2cm}
Let $T \in \N$ be such that $1/(2T+1) < 1/4 - \alpha/8$.

\begin{prop}\label{pqno1} 
Assume that $1 < \alpha < 2$. There exists $\theta >0$ such that, for n sufficiently large,
$$\Pro\left(\sum_{j:|a_{ij}|<b_{np}^\frac{T+1}{2T+1}}|a_{ij}| \leq \frac{1}{2}b_{np}^{\frac{3}{4} + \frac{\alpha}{8}} \right) \geq 1 - n\exp(-n^{\theta}).$$
\end{prop}
\begin{proof}

In order to prove the last proposition we introduce $Y_i^0 = p$, and for $k>1$:

$$Y_i^k \quad \equiv \quad \#(1 \leq j \leq p: |a_{ij}| \geq b_{np}^\frac{k}{2T+1}),$$
so that
\begin{equation}\label{coke}
\sum \indi_{\{j:|a_{ij}|<b_{np}^\frac{T+1}{2T+1}\}} |a_{ij}| \leq \sum_{k=0}^T Y_i^k b_{np}^{\frac{k+1}{2T+1}}.
\end{equation}

\begin{lem}\label{ponto}
Let $k \leq T$. There exists $\theta >0$ such that $ \Pro(Y_i^k \geq 2\E Y_i^k) \leq \exp(-p^\theta).$
\end{lem}

\begin{proof}[Proof of Lemma 7]
By definition of $Y_i^k$, and setting by convention $b_{np}^\frac{0}{2T+1}=0$, for all $1 \leq i \leq n$ and $0 \leq k \leq T$, $k \in \N$ we have $$\E Y_i^k=p\bar{F}(b_{np}^\frac{k}{2T+1}).$$ Also, using Chernoff's inequality, we have that 
\begin{equation}\label{Cher}
\Pro(Y_i^k \geq 2\E Y_i^k) \leq \exp(- \frac{1}{4}\E Y_i^k).
\end{equation} 
Thus, we can argue as follows:

First, replace the value $\E Y_i^k$ in (\ref{Cher}) and use the expression for $\bar{F}(x)$ given in (\ref{def1}). Hence, by (\ref{Luc}) there exists $\epsilon$ sufficiently small such that

\begin{eqnarray*}
\Pro(Y_i^k \geq 2\E Y_i^k) &\leq& \exp(-\frac{1}{4}p\bar{F}(b_{np}^\frac{k}{2T+1})) \leq \exp(-\frac{1}{4}p\bar{F}(b_{np}^\frac{T}{2T+1})) \\  
&\leq& \exp(-\frac{1}{4}pL(b_{np}^\frac{T}{2T+1})/b_{np}^\frac{T\alpha}{2T+1}) \leq \exp(-p/4b_{np}^\frac{T\alpha}{2T+1}b_{np}^\frac{T\epsilon}{2T+1}) \\
&\leq& \exp(-p/4(np)^{\frac{T}{2T+1}\frac{(\alpha + \epsilon)}{\alpha}}) \leq \exp(-p^\theta), \label{theta}
\end{eqnarray*}
where the last inequality is justified by the hypothesis  $\lim_{n\rightarrow\infty} \frac{p}{n} = \gamma$.
\end{proof}

\begin{lem} Assume that $1 < \alpha < 2$. Then,
\begin{equation}
\sum_{k=0}^T \E Y_i^k b_{np}^{\frac{k+1}{2T+1}} \leq \frac{1}{4}b_{np}^{\frac{3}{4} + \frac{\alpha}{8}}.
\end{equation}
\end{lem}

\begin{proof}[Proof of Lemma 8]
One has that
\begin{align}\label{pqno}
\E \left(\sum_{j:|a_{ij}|<b_{np}^\frac{T+1}{2T+1}}|a_{ij}|\right) &\leq \sum_{k=0}^T \E Y_i^k b_{np}^{\frac{k+1}{2T+1}} \leq \sum_{k=0}^T 2p\bar{F}(b_{np}^\frac{k}{2T+1})b_{np}^{\frac{k+1}{2T+1}} \nonumber\\
&\leq 2p\sum_{k=0}^T L(b_{np}^\frac{k}{2T+1})b_{np}^\frac{k+1}{2T+1}b_{np}^\frac{-k\alpha}{2T+1} \leq 2pb_{np}^\frac{1}{2T+1} \sum_{k=0}^T b_{np}^\frac{k(1+\delta-\alpha)}{2T+1}, \nonumber\\
\intertext{where $\delta > 0$ can be chosen such that $1+ \delta - \alpha < 0$. Therefore,}
\E \left(\sum_{j:|a_{ij}|<b_{np}^\frac{T+1}{2T+1}}|a_{ij}|\right) &\leq 2pb_{np}^\frac{1}{2T+1} \sum_{k=0}^\infty b_{np}^\frac{k(1+\delta-\alpha)}{2T+1} \leq pb_{np}^{\frac{1}{4} - \frac{\alpha}{8}} \nonumber\\
&\leq b_{np}^{\frac{1}{4} - \frac{\alpha}{8}+\frac{\alpha}{2}} \leq \frac{1}{4}b_{np}^{\frac{3}{4} + \frac{\alpha}{8}}.
\end{align}
\end{proof}

Combining the last two lemmas with (\ref{coke}), we get proposition $6$.

\end{proof}

%

Now we turn back to the sum of all terms which have absolute value between $b_{np}^{\frac{T+1}{2T+1}}$ and $b_{np}^{3/4 + \alpha/16}$. This sum is easier to handle since we have fewer entries. To simplify a little bit the notation below, put $\mu = \frac{T+1}{2T+1}$.

\begin{prop}\label{grande}
There exists $\kappa > 0$ such that  
$$\Pro\left(\exists i, \sum \indi_{\{j:b_{np}^{\mu}<|a_{ij}|<b_{np}^{\frac{3}{4}+\frac{\alpha}{16}}\}}|a_{ij}| \geq \frac{1}{2}b_{np}^{\frac{3}{4}+\frac{\alpha}{8}}\right) \leq \exp(-n^\kappa).$$
\end{prop}

\begin{proof}
\begin{eqnarray}
\Pro\left(\exists i, \sum \indi_{\{j:b_{np}^{\mu}<|a_{ij}|<b_{np}^{\frac{3}{4}+\frac{\alpha}{16}}\}}|a_{ij}| \geq \frac{1}{2}b_{np}^{\frac{3}{4}+\frac{\alpha}{8}}\right) &\leq& n\Pro(\# \{ j:b_{np}^{\mu}<|a_{ij}| \} \geq \frac{1}{2}b_{np}^{\frac{\alpha}{16}}) \nonumber \\
&\leq& n n^{\frac{1}{2}b_{np}^{\frac{\alpha}{16}}} \bar{F}(b_{np}^{\mu})^{\frac{1}{2}b_{np}^{\frac{\alpha}{16}}} \nonumber \\
&\leq& n \left(\frac{nL(b_{np}^{\mu})}{b_{np}^{\mu\alpha}}\right)^{\frac{1}{2}{b_{np}^{\frac{\alpha}{16}}}} \nonumber\\ 
&\leq& \exp(-n^\kappa),
\end{eqnarray}
for some sufficiently small $\kappa > 0$ since $\mu > 1/2$. 
\end{proof}

Let us finish the proof of statement $(b)$ when $1 < \alpha < 2$. By part (a) we know that, with probability going to $1$ as $p$ goes to infinity, there is at most one term in each line that exceeds $b_{np}^{\frac{3}{4}+\frac{\alpha}{16}}$ in absolute value.  So it is enough to consider the sum of all entries less than or equal to $b_{np}^{\frac{3}{4}+\frac{\alpha}{16}}$ and prove that in fact the probability that this sum is less than $b_{np}^{\frac{3}{4}+\frac{\alpha}{8}}$ goes to one as $p$ tends to infinity. We proved  this statement in two parts, analyzed in proposition $\ref{pqno1}$ and proposition \ref{grande} respectively.   

\vspace{0.5cm}

\textit{Case $0<\alpha \leq 1$:}

\vspace{0.2cm}

We repeat the same argument and computation used in the other case. We begin by proving the counterpart of Proposition $\ref{pqno1}$ in this case.

\begin{prop}\label{ger12} There exists $\theta > 0$ such that for $n$ large enough
$$ \Pro \left(  \sum_j \indi_{\{ j:|a_{ij}|<b_{np}^\frac{T+1}{2T+1} \}} |a_{ij}| \leq \frac{1}{2}b_{np}^{\frac{3}{4} + \frac{\alpha}{8}} , \forall i \right) \geq 1 - n\exp(-n^{\theta}). $$
\end{prop}
\begin{proof}
Lemma $\ref{ponto}$ is valid when $0 < \alpha < 2$ so that it is enough to prove
\begin{lem} Let $ 0 < \alpha \leq 1$ and $T$ such that $\frac{1}{2T+1} < \frac{\alpha}{8}$. Then,
\begin{equation}
\sum_{k=0}^T \E Y_i^k b_{np}^{\frac{k+1}{2T+1}} \leq \frac{1}{2}b_{np}^{\frac{3}{4} + \frac{\alpha}{8}}.
\end{equation}
\end{lem}
\begin{proof} One has that

\begin{eqnarray*}
\sum_{k=0}^T \E Y_i^k b_{np}^{\frac{k+1}{2T+1}} &\leq& \sum_{k=0}^T 2p\bar{F}(b_{np}^\frac{k}{2T+1})b_{np}^{\frac{k+1}{2T+1}} \leq 2p\sum_{k=0}^T L(b_{np}^\frac{k}{2T+1})b_{np}^\frac{k+1}{2T+1}b_{np}^\frac{-k\alpha}{2T+1} \\
&\leq& 2pb_{np}^\frac{1}{2T+1} \sum_{k=0}^T b_{np}^\frac{k(1+\delta-\alpha)}{2T+1} \leq b_{np}^{\frac{1}{2} + \frac{\alpha}{4}} \leq \frac{1}{2}b_{np}^{\frac{3}{4} + \frac{\alpha}{8}}.
\end{eqnarray*}  
\end{proof}
Using Proposition \ref{ger12} and Proposition \ref{grande} as before it is easy to prove statement (b) of Lemma \ref{lem1}.
\end{proof}
\emph{Proof of (c):}
As one can easily see, the proof of item (c) of the lemma is identical to the proof of part (b) up to a permutation of p's and n's. 
\end{proof}

\begin{rem}\label{rem1}
Recalling definition (\ref{norm}), statement (b) of Lemma \ref{lem1}, Proposition \ref{pqno1} and Proposition \ref{ger12} show that for every $\epsilon > 0$ there exists $n_o(\epsilon)$ and $\theta>0$ such that for all $n>n_o$ one has
$$\Pro \left(|\frac{\|A_n\|_{\infty}}{\max_{ij} |a_{ij}|} - 1| > \epsilon \right) \leq n\exp(-n^{\theta}).$$
By part (c) of Lemma \ref{lem1}, the same is valid if one replaces $\|A_n\|_{\infty}$ above by $\|A_n\|_{1} \equiv \sup_{j} \sum_{i=1}^{p} |a_{ij}|$.
\end{rem}

\begin{rem}\label{rem2}
From now on, if $X_n$ and $Y_n$ are two sequences of random variables defined on the same probability space, we will use the notation $X_n = Y_n(1 + o(1))$ to indicate that for all $\epsilon >0$, the probability $\Pro(|\frac{X_n}{Y_n} - 1| > \epsilon)$ goes to $0$ as $n$ goes to infinity, i.e the ratio $\frac{X_n}{Y_n}$ converges in probability to $1$.
\end{rem}

\subsection{Proof of Theorem \ref{thm1}}

\begin{proof}[Proof of Corollary \ref{cor2}] We begin by the proof of  Corollary \ref{cor2}.
The main thing to show is that

\begin{equation}
\frac{\lambda_1}{a_{i_1j_1}^{2}}\label{prob} \longrightarrow 1 ,
\end{equation} 
in probability  as $n$ tends to infinity (we recall that $a_{i_1j_1} = \max |a_{ij}|$). In fact, if we assume (\ref{prob}), extreme value theory for iid random variables tell us that 

\begin{equation}
\lim_{n\rightarrow\infty} \Pro (\frac{a_{i_1j_1}^{2}}{b_{np}^2} \leq x ) = \exp (-x^{-\frac{\alpha}{2}})\label{prob2},
\end{equation} 
so (\ref{prob}) and (\ref{prob2}) will imply Corollary $\ref{cor2}$. Thus, our task is to prove (\ref{prob}) and the idea  is as follows: Given $\epsilon > 0$, we want to show that for $n$ sufficiently large we have   

\begin{equation}\label{geq}
\lambda_1 \geq a_{i_1j_1}^{2} (1+o(1)),
\end{equation}

\begin{equation}\label{leq}
\lambda_1 \leq a_{i_1j_1}^{2} (1+o(1)),
\end{equation}
with probability greater than $1 - \epsilon$. The main tool used to prove both equations will be Lemma \ref{lem1}, and we will start with the easiest inequality, (\ref{geq}).

Since for all unit vectors $v$ we have the bound $\left\langle X_nv,v \right\rangle \leq \lambda_1$, our task is the following: we must find a suitable vector that gives us (\ref{geq}). Therefore, let $(i_{1},j_{1})$ be the position of $a_{i_1j_1}$ in $A_n$ as the notation suggests. If one takes $v=(0,\ldots,0,1,0,\ldots,0)$ where the sole non-zero entry of $v$ is in the position $i_1$, the vector $v$  will do the job. In fact, 
$$\left\langle X_nv,v \right\rangle \quad = \quad \sum_{j=1}^{p} a_{i_{1}j}^{2} \quad = \quad a_{i_1j_1}^2 + \sum_{j=1, j\neq j_1}^{p} a_{i_{1}j}^{2} = a_{i_1j_1}^2 (1 + o(1)), $$ by part (b) of Lemma \ref{lem1} and the fact that $a_{i_1j_1}$ is the maximum of $np$ iid random variables. This proves (\ref{geq}).

To obtain (\ref{leq}) we first recall the definition of $$\|X_n\|_{\infty} \equiv \sup_{i} \sum_{j=1}^{n} |X_{ij}|.$$ The eigenvector equation for $\lambda_1$
 $$\sum_{j=1}^{n} X_{ij}v_{j} = \lambda_1  v_{i}$$
implies that

\begin{align*}
&\lambda_1|v_i| \leq \sum_{j=1}^n |X_{ij}||v_j| \leq \sup_l |v_l|\sum_{j=1}^n |X_{ij}|\\
\intertext{so,}
&\lambda_1\sup_i |v_i| \leq \left(\sup_l|v_l|\right) \sup_i \sum_{j=1}^n |X_{ij}|.
\end{align*} 
Therefore, $\|X_n\|_{\infty}$ is an upper bound for $\lambda_1$.

Hence, with probability going to one,
\begin{eqnarray*}
\lambda_1 \leq \|X_n\|_{\infty} &\leq& \sup_i\left\{\sum_{j=1}^{n} \sum_{k=1}^{p}|a_{ik}||a_{jk}|\right\} \leq \sup_i \left\{\sum_{k=1}^{p}|a_{ik}|\sum_{j=1}^{n} |a_{jk}|\right\} \\
&\leq& \sup_i \left\{\sum_{k=1}^{p}|a_{ik}|\right\}\sup_{l}\left\{\sum_{j=1}^{n} |a_{jl}|\right\} \leq \|A_n\|_{\infty}\|A_n\|_{1\rightarrow1} \\
&\leq& a_{i_1j_1}^2 (1 + o(1)),
\end{eqnarray*}
where the last inequality comes from Remark \ref{rem1}. \end{proof}

\vspace{1.5cm}

\begin{proof}[Proof of Theorem \ref{thm1}]  It is enough to show that for any finite $k$ we have for all $1 \leq l \leq k$
\begin{equation}
\lim_{p \rightarrow \infty} \Pro (\dfrac{1}{b_{np}^2}\lambda_l \leq x) =  \lim_{p \rightarrow \infty} \Pro (\dfrac{1}{b_{np}^2}a_{i_lj_l}^2 \leq x),
\end{equation}
where $a_{i_lj_l}$ is the $l-$th term of the sequence $|a_{ij}|$ in the decreasing order.

Let $e_1, \ldots, e_p$ be the standard orthonormal basis of $\R^p$. If we compute $X_n.e_{i_l}$ we get:
\begin{equation}\label{r_l}
X_ne_{i_l} = \sum_{i=1}^n X_{ii_l}e_i = X_{i_li_l}e_{i_l} + r_l,
\end{equation}  
for some vector $r_l$ in $\R^n$.

Also, since $X_n$ is symmetric, one can find a orthogonal matrix $U$ and a diagonal matrix $D$ such that $X_n = UDU^{-1}$. Now, suppose that $X_{i_li_l}$ is not an eigenvalue of $D$. Then $D-X_{i_li_l}I$ is invertible and one can use  equation (\ref{r_l}) to get: 
\begin{equation}
1 = \|e_{i_l}\| = \|U(D-X_{i_li_l}I)^{-1}U^{-1}r_l\| \leq \|U\| \|U^{-1}\| \|(D-X_{i_li_l}I)^{-1}\|\|r_l\|,
\end{equation}
which implies
\begin{equation}
\min_i|\lambda_i - X_{i_li_l}| \leq \|r_l\|,
\end{equation}
so there exists an eigenvalue $\lambda$ of $X_n$ such that
\begin{equation}\label{boundr}
|\lambda - X_{i_li_l}| \leq \|r_l\|.
\end{equation}
If $X_{ii_l}$ is an eigenvalue of $D-X_{i_li_l}I$, (\ref{boundr}) is clearly satisfied.

We now know, by Lemma \ref{lem1}, that $X_{i_li_l} = a_{i_lj_l}^2(1+o(1))$. Therefore, if we manage to prove that $r_l$ has a norm that is negligible with respect to $a_{i_lj_l}^2,$ we will be able to say that $X_n$ has eigenvalues $a_{i_lj_l}^2(1+o(1))$, $1 \leq l \leq k$ for any finite $k$.

Bounding the norm of $r_l$, one gets

\begin{align*}
\|r\|= \left(\sum_{i=1, i\neq i_l}^n X_{ii_l}^2\right)^{1/2} &\leq \sum_{i=1, i\neq i_l}^n |X_{ii_l}| \\
&\leq \sum_{i=1, i\neq i_l}^n \sum_{k=1}^p |a_{ik}||a_{i_lk}| \\
&= \sum_{k=1}^p |a_{i_lk}| \sum_{i=1, i\neq i_l}^n |a_{ik}| \equiv S_1. 
\end{align*}

We cannot estimate $S_1$ directly as we did in part (a) but if we define 
\begin{equation}
S_2 = \sum_{k=1, k \neq j_l}^p |a_{i_lk}| \sum_{i=1, i\neq i_l}^n |a_{ik}|,
\end{equation}
then 
\begin{equation}
S_2 \leq \left(\sup_k \sum_{i=1, i\neq i_l}^n |a_{ik}| \right) \sum_{k=1, k \neq j_l}^p |a_{i_lk}|,
\end{equation}
which tells us that $S_2$ is negligible with respect to $a_{i_lj_l}^2$, again by Lemma \ref{lem1}. Now, 
\begin{equation}\label{diff}
S_1 - S_2= |a_{i_lj_l}|\sum_{k=1, k \neq i_l}^n |a_{i_lk}|.
\end{equation}
which is also negligible with respect to $a_{i_lj_l}^2$ since $ \sum_{k=1, k \neq j_l}^n |a_{i_lj_l}|$ is negligible by a direct application of part (b) of Lemma \ref{lem1}. Hence, 
$S_1$ is also negligible with respect to $a_{i_1j_1}^2$.

We now know that $a_{i_lj_l}^2(1+o(1))$, $1 \leq l \leq k$ are eigenvalues of $X_n$. However, this does not imply that they are exactly the $k$ top eigenvalues. At the moment this is true only for the maximum, by Corollary (\ref{cor2}), and we need to check it for $1 < l \leq k$. In other words, what we get from the last statement is that for all $1 < l \leq k$, for $p$ large enough: 

\begin{equation}
\lambda_{l} \geq a^2_{i_lj_l}(1+o(1)),
\end{equation}
and we need to prove the reverse inequality. To achieve our goal, we consider the compression of the matrix $X_n$ step by step, i.e., we cut from $A_n$ the row $i_1$ and from $A_n^t$ the column $i_1$ and then we compute their product. By part (a) of Lemma \ref{lem1}, the entry $a_{i_2j_2}$ is still in the matrix and the product $X^{(2)}_n$ is just the matrix $X_n$ without the row and column $i_1$. Now we know by the Cauchy Interlacing Theorem, see \cite{Bha}, Corollary $3.1.5$,  that if $\kappa_1 \geq \ldots \geq \kappa_{p-1}$ are the eigenvalues of $X^{(2)}_n$, we have:

\begin{equation}\label{cauchy}
\lambda_{1} \geq \kappa_1 \geq \lambda_2 \geq \kappa_2 \geq \ldots \geq \kappa_{p-1} \geq \lambda_p.
\end{equation}

Combining (\ref{cauchy}) with Corollary \ref{cor2} applied for the matrix $X^{(2)}_n$, we get the desired inequality for $\lambda_2$. Repeating the same argument for all $1 \leq l \leq k$, the proof is complete.

\end{proof} 


\section{Wigner matrices when $2 \leq \alpha < 4$}\label{wigner}

Now we consider the symmetric random matrix $(a_{ij})_{i,j=1}^n$ where the entries of $A_n$ are centered i.i.d. and satisfy (\ref{def1}) with $2 \leq \alpha < 4$. We treat separately the case where $\alpha = 2$.

\subsection{Truncation}
The main difference between the proof of this section to the previous one is that here we should care about the contribution given by the bulk of the spectra, that is we should control in some way the smaller entries of the matrix $A_n$ and then proceed as before. Thus, to investigate the behavior of the largest eigenvalue $\lambda_1$, we split the above random matrix as follows. Let $\beta$ be such that
\begin{equation}\label{beta}
\frac{1}{\alpha} < \beta < \frac{2(8-\alpha)}{\alpha(10-\alpha)}, 
\end{equation}
and we define
\begin{equation}\label{a2}
A_1=(A_{ij}\indi_{|a_{ij}|\leq n^{\beta}})_{i,j=1}^n, \quad A_2 = A_n - A_1.
\end{equation} 

Since $\beta < \frac{2}{\alpha}$, it is also clear that with probability going to 1, the largest entry of $A_n$ is the largest entry of $A_2$.  The condition that $\beta > 1/ \alpha$ is assumed to guarantee that we can study the asymptotic behaviour of the eigenvalues of $A_2$ in a similar way of the previous section. On the other hand, the condition that $\beta < \frac{2(8-\alpha)}{\alpha(10-\alpha)}$ is assumed to guarantee that the spectrum of $A_1$, properly normalized, remains bounded. 

\subsection{Bounding the spectrum of $A_1$}\label{comb}

We first investigate the behavior of the largest eigenvalue of $A_1$ referring the reader to the results of \cite{Bai} and \cite{Yin}. These papers deal with the case of random Wigner matrices with the presence of a finite fourth moment and they prove boundedness of the spectra. Fix some $\epsilon>0$ such that 
$$\epsilon <\min \big \{\frac{1}{\alpha} -\frac{1}{4}, \frac{1}{\alpha} -\frac{\beta}{2} ,\frac{1}{16}(\frac{8}{\alpha}-1 -\beta(5-\frac{\alpha}{2}))\big \}.$$
Here, we will prove that the largest eigenvalue and the smallest eigenvalue of $\frac{1}{n^{2/\alpha -\epsilon}}A_1$ are bounded on a set of probability arbitrarily close to $1$. In this direction, our main result in this subsection will be:

\begin{prop}\label{propw}
Let $s_n$ be some sequence going to infinity in such a way that $\log n << s_n<<n^{\gamma}$ where $0<\gamma \leq \min \{ \frac{1}{8}(\frac{8}{\alpha}-1 -\beta(5-\frac{\alpha}{2})), \frac{1}{2\alpha}-\frac{\beta}{4}\}$. Then there exists a constant $C>0$ such that $\E\left(\text{Tr} (\frac{A_1}{n^{2/\alpha -\epsilon}})^{2s_n}\right) < C (2\sigma)^{2s_n}\frac{n}{s_n^{3/2}}. $
\end{prop}

Before giving the proof, we indicate how to use Proposition \ref{propw} to deduce the desired result, that is the boundedness of the largest eigenvalue of $\frac{1}{n^{2/\alpha -\epsilon}}A_1$. 
We have
\begin{eqnarray*}
\Pro\left (\lambda_{1}(\frac{1}{n^{2/\alpha -\epsilon}}A_1) \geq 4 \sigma\right ) &\leq& \frac{E\left (\lambda_{1}(\frac{1}{n^{2/\alpha -\epsilon}}A_1)^{2s_n}\right )}{(4\sigma)^{2s_n}} \\
&\leq& \frac{\E\left (\text{Tr}(\frac{A_1}{n^{2/\alpha -\epsilon}})^{2s_n}\right )}{(4\sigma)^{2s_n}} \\
&\leq& \exp(-\eta s_n) 
\end{eqnarray*}
for some constant $\eta > 0$, proving that $\lambda_{1}(\frac{1}{n^{2/\alpha -\epsilon}}A_1)$ is bounded in probability. By symmetry, one gets the same result for the smallest eigenvalue.

\begin{proof}[Proof of Proposition \ref{propw}]
To estimate $\E\left (\text{Tr}(\frac{A_1}{n^{2/\alpha -\epsilon}})^{2s_n}\right)$, we use the moment method. Developing the expectation, we have that
\begin{equation}\label{expec}
\E\left (\text{Tr}(A_1)^{2s_n}\right )= \sum_\mathscr{P} \E \hat{a}_{i_0i_1}\hat{a}_{i_1i_2}\hat{a}_{i_2i_3}\hat{a}_{i_3i_4}\hat{a}_{i_4i_5}\ldots \hat{a}_{i_{2s_n-2}i_{2s_n-1}}\hat{a}_{i_{2s_n-1}i_0},
\end{equation}
where $\hat{a}_{ij}=a_{ij}\indi_{\{|a_{ij}|\leq n^{\beta}\}}$ and $\mathscr{P}$ denotes the set of all closed paths $P=\{i_0,i_1,\ldots,i_{2s_n-1},i_0\}$ with a distinguished origin, in the set $\{1,2,\ldots,n\}$. 

The following two lemmas are a direct consequence of \cite{Fel}, chapter $VIII.9$, Theorem $2.23$.

\begin{lem}\label{center} Let $C_n = \E(\hat{a}_{ij}) = \E(a_{ij}\indi_{|a_{ij}|\leq n^{\beta}})$. Then for $n$ large enough,  $|C_n| \leq L(n^{\beta})n^{\beta(1-\alpha)}$ where $L$ is defined in (\ref{def1}).
\end{lem}

\begin{lem}\label{moments} For any $k \geq 2$, let $D_n^{2k} = \E(\hat{a}_{ij}^{2k}) = \E(a_{ij}^{2k}\indi_{|a_{ij}|\leq n^{\beta}})$. Then for $n$ large enough, there exists a slowly varying function $l_0$ such that $D_n^{2k} \leq l_0(n^{\beta})n^{\beta(2k-\alpha)}$.
\end{lem}

Thus it follows that 
\begin{equation}
|\lambda_1(A_1-\E A_1)-\lambda_1(A_1)|\leq |C_n|n \leq L(n^{\beta})n^{\beta(1-\alpha)+1}, 
\end{equation}
so one can write
\begin{equation}\label{tricen}
\lambda_{1}(\frac{1}{b_{n}}A_n) \leq \lambda_{1}(\frac{1}{b_{n}}(A_1 - \E A_1)) + \lambda_{1}(\frac{1}{b_{n}}(A_2)) + \frac{1}{b_{n}}|\lambda_1(A_1-\E A_1)-\lambda_1(A_1)|,
\end{equation}
where the last term tends to zero as $n$ tends to infinity since $\beta(1-\alpha)+1-\frac{2}{\alpha}<0$.
Thus, we may assume that $A_1$ is also centered, i.e, the truncated variables $\hat{a}_{ij}$ are centered.

Now we move back to equation $(\ref{expec})$, to compute the expected value of the trace of $\frac{A_1^{2s_n}}{n^{(4/\alpha -2\epsilon)s_n}}$. The first step after the centering is to consider the contribution of even paths, i.e. paths such that each edge occurs an even number of times.  

We refer to the paper of A. Soshnikov \cite{Sos2} for most of the details and further notation. To each path $P= i_o\rightarrow i_1 \rightarrow i_2 \ldots \rightarrow i_{2s_{n-1}} \rightarrow i_0$, we first associate a set of $s_n$ ``marked instants'' as follows. We read the edges of $P$ successively. The instant at which an edge $i\rightarrow j$ is read is then said to be marked if up to that moment (inclusive) the edge $(i,j)$ was read an odd number of times. Other instants are said to be unmarked.
Now, the number of possible arrangements of marked/unmarked instants in a path of length $2s_n$ is equal to the number of Dick paths, i.e., the number of simple random walks of length $2s_n$, starting and ending at $0$, and conditioned to remain in the positive quadrant. The number of Dyck paths is known to be the Catalan number: $\frac{(2s_n)!}{s_n!(s_n+1)!}$. 

We say that a vertex is marked if it occurs at a marked instant. For any $0\leq k \leq s_n$, we then define $N_{k}$ to be the subset of vertices in $\{1, \ldots, n\}$ occurring $k$ times as a marked vertex. Any vertex belonging to $N_k$ is said to be a vertex of self-intersection of type $k$. For any $0\leq k\leq s_n,$ we denote by $n_k$ the cardinality of $N_k$ and call $(n_o, n_1, \ldots, n_{s_n})$ the type of $P.$
Note that all the vertices that appear in $ P$ belong to $N_k, k \geq 1$ except possibly the origin $i_o$.
 
Then, it is easy to see that $\{1, \ldots, n\}$ splits as the disjoint union of the sets $N_o, N_1, \ldots , N_{s_n}$. From these definitions, one can see that
\begin{equation}\label{graph}
\sum_{k=0}^{s_n} n_k = n, \quad \text{and} \quad \sum_{k=0}^{s_n} kn_k = s_n.
\end{equation}

To estimate the number of possible paths and their contribution to the expectation, we proceed as follows.
We first determine the set of marked instants and the type of the path. Then, we assign labels chosen in $\{1, \ldots, n\}$ to each marked instant and to the origin of the path. Finally, we assign labels to each unmarked instant and consider the expectation of the corresponding path.\\
Given the set of marked instants and the type of the path $(n_o, n_1, \ldots, n_{s_n})$, one has exactly $\frac{s_n!}{\prod_{k=2}^{s_n} (k!)^{n_k}}$ ways to distribute the marked instants into the possible classes of self-intersection. The number of ways to distribute the vertices of $\{1, \ldots, n\}$ into the set of possible classes $N_o, N_1, \ldots, N_{s_n}$ and determine the origin of the path is at most 
$\frac{n!}{n_0!n_1!\ldots n_{s_n}!}n.$ This is because the origin is in general a non-marked vertex.
There now remains to give an upper bound on the number of ways to determine vertices at unmarked instants, that is fill in the blanks of the path. It was proved in \cite{Sos2} that the number of ways to assign labels at unmarked instants is not greater than $\prod_{k=2}^{s_n}(2k)^{kn_k}$. Indeed, the number of possible ways to determine the right endpoint of an edge starting from a vertex of type $k$ at an unmarked instant is at most $2k.$

To consider the expectation of a path $ P$ of type $(n_o, n_1, \ldots, n_{s_n})$, we will need the following Lemma. 
\begin{lem}\label{Lem: expectation}
Consider an even path of type $(n_0,n_1,\ldots,n_{s_n})$. One has
\begin{equation}
\E \left(\hat{a}_{i_0i_1}\hat{a}_{i_1i_2}\hat{a}_{i_2i_3}\hat{a}_{i_3i_4}\hat{a}_{i_4i_5}\ldots \hat{a}_{i_{2s_n-2}i_{2s_n-1}}\hat{a}_{i_{2s_n-1}i_0}\right) \leq \sigma^{2s_n}\prod_{i\geq 2}\left (l_o(n^{\beta})^{i}n^{\beta (2i-(\alpha/2-1))}\right)^{n_i}. 
\end{equation}
\end{lem}
\begin{proof}
Assume that a non-oriented edge $(ij)$ is seen $2l(ij)$ times. We denote by $l(i;ij)$ (resp. $l(j;ij)$) the number of times $i$ (resp. $j$) is a marked vertex in $(ij)$. We also set $L(ij)=\max \{l(i;ij), l(j;ij)\}$ and $L'(ij)=\min \{l(i;ij), l(j;ij)\}$. 

First, using Lemma \ref{moments}, we deduce that
\begin{eqnarray}\label{ups}
&\prod_{(ij): l(ij)>1}\mathbb{E}\hat{a}_{ij}^{2l(ij)}&\leq \prod_{(ij): l(ij)>1}l_o(n^{\beta})n^{\beta (2l(ij) -\alpha)}\crcr
&&\leq \prod_{(ij): l(ij)>1}l_o(n^{\beta})n^{\beta (2L(ij)+\frac{2-\alpha}{2}+2L'(ij)-2 +\frac{2-\alpha}{2})}
\end{eqnarray}

Second, we change the product in (\ref{ups}) over all edges to a product over all vertices. In fact, one can associate to each edge occurring $4$ times at least one marked occurrence of a vertex of self-intersection. To deal with vertices where $L'(ij)=L(ij)=1$, we say that $L'(ij)$ is associated to the vertex which has the smallest multiplicity in the path. Using the fact that the number of marked occurrences of any vertex in edges seen at least $4$ times cannot exceed the type of the vertex and (\ref{graph}), we get:

\begin{eqnarray}
&\prod_{(ij): l(ij)>1}\mathbb{E}\hat{a}_{ij}^{2l(ij)}&\leq \prod_{(ij): l(ij)>1}l_o(n^{\beta})n^{\beta (2l(ij) -\alpha)}\crcr
&&\leq n^{\beta \sum_{(ij): l(ij)>1}(2L(ij)+\frac{2-\alpha}{2})+\beta \sum_{(ij): l(ij)>2}(2L'(ij)-2 +\frac{2-\alpha}{2})} \prod_{k\geq 2}l_o(n^{\beta})^{kn_k}\crcr
&& \leq \prod_{i\geq 2}\left (l_o(n^{\beta})^{i}n^{\beta (2i-(\alpha-2)/2)}\right)^{n_i}.
\end{eqnarray}

\end{proof}

\begin{rem}\label{Ruzma}
Lemma \ref{Lem: expectation} plays the role of Formula $4.7$ in \cite{Ruz} to bound the contribution of a single path of type $(n_0,n_1,\ldots,n_{s_n})$. Here is where we cannot understand the arguments of \cite{Ruz}. Indeed, with the notation of \cite{Ruz}, we believe that formula $4.7$,
$$\left(\E \prod_{u=0}^{2s_n-1} \xi_{i_ui_{u+1}}| \Omega_{1-\epsilon_n} \right) W_n \leq \frac{1}{4^{s_n}}4^r \prod_{k = 3}^{\frac{p}{2}} (2kC)^{kn_k} \prod_{k=\frac{p}{2}}^{s_n}(2k \Lambda_n^2)^{kn_k}, $$
should be written as 
$$\left(\E \prod_{u=0}^{2s_n-1} \xi_{i_ui_{u+1}}| \Omega_{1-\epsilon_n} \right) W_n \leq \frac{1}{4^{s_n}}4^r \prod_{k = 3}^{\frac{p}{4}} (2kC)^{kn_k} \prod_{k=\frac{p}{4}}^{s_n}(2k \Lambda_n^4)^{kn_k}, $$
since, as argued in Lemma \ref{Lem: expectation}, an edge seen $p$ times implies the occurrence of a marked vertex of type at least $p/4$ but not necessarily of type $p/2$. Mutatis mutandis, the other arguments of Ruzmaikina carry out to show that the universal Tracy-Widom limit holds if $\alpha > 36$.
\end{rem}

We now call $Z_e$ the contribution of all even paths of type $(n_o,n_1, \ldots, n_{s_n})$ to the expectation $\mathbb{E} \text{Tr} \left ( \frac{A_1}{n^{2/\alpha -\epsilon}}\right)^{2s_n}.$
Writing $$\mathbb{E} \text{Tr} \left ( \frac{A_1}{n^{2/\alpha -\epsilon}}\right)^{2s_n}=
\mathbb{E}\text{Tr} \left ( \frac{A_1}{\sqrt n}\right)^{2s_n}\left (\frac{1}{n^{2/\alpha-1/2 -\epsilon}}\right)^{2s_n},$$
we deduce that $Z_e$ is bounded by:

\begin{eqnarray}
&Z_e&\leq \frac{1}{n^{s_n(4/\alpha-2\epsilon)}}\frac{n!}{n_0!n_1!\ldots n_{s_n}!}n\frac{(2s_n)!}{s_n!(s_n+1)!}\frac{s_n!}{\prod_{k=2}^{s_n} (k!)^{n_k}}\prod_{k=2}^{s_n}(2k)^{kn_k}\crcr
&&\times\sigma^{2s_n}\prod_{k=2}^{s_n}\left (l_o^{k}n^{\beta (2k-(\alpha/2-1))})\right)^{n_k} \nonumber \\
&\leq& \frac{1}{n^{(4/\alpha-1-2\epsilon)\sum_{i\geq 1}in_i}}\frac{n\ldots(n_0+1)}{n^{s_n}}n\frac{(2s_n)!}{s_n!(s_n+1)!}\frac{1}{\prod_{k=2}^{s_n}n_k!}\frac{s_n^{s_n-n_1}}{\prod_{k=2}^{s_n} (k!)^{n_k}}\prod_{k=2}^{s_n}(2k)^{kn_k}\crcr
&&\times \sigma^{2s_n}\prod_{k=2}^{s_n}\left (l_o^{k}n^{\beta (2k-(\alpha/2-1))}\right)^{n_k} \nonumber \\
&\leq& 
\frac{1}{n^{(4/\alpha-2\epsilon-1)\sum_{i\geq 1}in_i}}\frac{(2s_n)!}{s_n!(s_n+1)!}n\frac{1}{\prod_{k=2}^{s_n}n_k!}\frac{s_n^{s_n-n_1}}{n^{s_n+n_0-n}}\frac{1}{\prod_{k=2}^{s_n}(ke^{-1})^{kn_k}}\prod_{k=2}^{s_n}(2k)^{kn_k}\crcr
&& \times \sigma^{2s_n}\prod_{k=2}^{s_n}(l_o^{k}n^{\beta (2k-(\alpha/2-1))})^{n_k} \nonumber \\
&\leq&
\frac{(2s_n)!}{s_n!(s_n+1)!}n\sigma^{2s_n}\prod_{k=2}^{s_n}\frac{1}{n_k!}\left[\frac{l_o^{k}s_n^kn^{\beta(2k-(\alpha/2-1))}}{n^{4k/\alpha -2k\epsilon-1}} \right]^{n_k}\frac{1}{n^{n_1(4/\alpha-2\epsilon -1)}}.
\end{eqnarray}
First, we consider the contribution of simple paths, that is paths of type $(n_o,n_1, 0, \ldots, 0).$
We denote $Z_{e,s}$ this contribution. Then 
$$Z_{e,s}\leq \frac{(2s_n)!}{s_n!(s_n+1)!}n\sigma^{2s_n}\frac{1}{n^{s_n(4/\alpha-1-2\epsilon) }}.$$
This follows from the fact that in simple even paths, any edge is seen twice and the choice of the origin and marked vertices determines the path.\\
We next turn to the contribution of paths with self-intersections, which we denote by $Z_{e,i}$.
Now, if we take the sum  over all non-negative integers $n_2, n_3, \ldots n_{s_n}$ such that 
\begin{equation}
\sum_{k=2}^{s_n}n_k > 0,
\end{equation}
we have that:

\begin{equation}
Z_{e,i}\leq \frac{(2s_n)!}{s_n!(s_n+1)!}n\sigma^{2s_n}\sum_{n_2,\ldots,n_{s_n}}\prod_{k=2}^{s_n}\frac{1}{n_k!}\left[\frac{l_o^{k}s_n^kn^{\beta(2k-(\alpha/2-1))}}{n^{4k/\alpha -2k\epsilon-1}} \right]^{n_k}\label{majoini}.
\end{equation}
Now as $\epsilon< 2/\alpha-\beta$ and $s_n n^{2\beta}<<n^{4/\alpha}$, we deduce that 
\begin{eqnarray}&(\ref{majoini})&\leq \frac{(2s_n)!}{s_n!(s_n+1)!}n \sigma^{2s_n} \sum_{M>0}\frac{1}{M!}\left[\frac{Cs_n^2n^{\beta(4-(\alpha/2-1))}}{n^{8/\alpha -4\epsilon-1}}\right]^{M}\crcr
&&\leq C\frac{(2s_n)!}{s_n!(s_n+1)!}n \sigma^{2s_n}\frac{ l_o^2s_n^2n^{\beta(5-\alpha/2)) }}{n^{8/\alpha -4\epsilon-1}}\crcr
&&\leq  C\frac{(2s_n)!}{s_n!(s_n+1)!}n \sigma^{2s_n} l_o^2s_n^2n^{\beta(5-\alpha/2)+1-8/\alpha +4\epsilon}\crcr
&& =o(1)\frac{(2s_n)!}{s_n!(s_n+1)!}n \sigma^{2s_n}.
\end{eqnarray}
In the last line, we have used the fact that $s_n<<n^{(\frac{8}{\alpha}-1 -\beta(5-\frac{\alpha}{2}))/8}.$
\paragraph{}Now, we will see that this is also true for a path such that an edge occurs an odd number of times, proving proposition \ref{propw}. The necessary tools are the gluing and the insertion procedures developed in \cite{SosPec}. We refer to this article for details and notation.  

In \cite{SosPec}, one can prove that given a path $P$ of length $2s_n$ with $2l$ non-returned edges, it is possible to construct a sequence $(P_0, P_1, P_2, \ldots , P_J)$, $1 \leq J \leq 2l$, of subpaths of $P$  such that the concatenated path $W = \bigcup_{i=0}^{J} P_i$, i.e. the path defined as if we read $P_i$ in order, has length $2s_n - 2l$ and belongs to one of the following classes:
\begin{enumerate}
\item[\textbf{A}] $W$ is a closed even path.
\item[\textbf{B}] $W$ is a sequence of $I \leq 2$ closed even paths where each origin is a marked vertex of $P$. 
\item[\textbf{C}] $W$ is a sequence of $I \leq 2$ paths where each origin is a marked vertex of $P$ and the union of these paths has only even edges.
\end{enumerate}
  
The surgery in $P$ consists only to remove the last occurrence of odd edges and reorder the remaining subpaths with the possibility of choosing the direction in which each subpath is read.
To estimate the contribution of odd paths, one can reverse the above procedure, defining an onto map from paths of classes A, B, C to the set of odd paths. This was done in \cite{SosPec} and we concisely describe the method.

In case A, the simplest one, given a closed path $W$ of length $2s_n - 2l$ one needs only to choose $J$ vertices to split $W$, choose the order and direction of each subpath and how to assign and insert the $2l$ unreturned edges. Also, since $\hat{a_{ij}}$ are bounded by $n^{\beta}$, adding these $2l$ repetitions we multiply the contribution of the original path at most by $n^{2l\beta}$. Briefly, the contribution of paths with $2l$ odd edges such that $W(P) \in A$ can be bounded by: 

\begin{equation}\label{odd1A}
C_1\frac{(2s_n - 2l)!}{(s_n-l)!(s_n-l+1)!}n\sigma^{2(s_n-l)} \sum_{J=1}^{2l} \binom{2s_n - 2l}{J} J! 2^J \binom{2l}{J}\frac{(2s_n - 2l)!}{(2s_n-4l +J)!}\left (\frac{n^\beta}{n^{2/\alpha- \epsilon}}\right )^{2l},
\end{equation}
which can be bounded by:
\begin{equation}\label{odd2A}
C_2\frac{(2s_n - 2l)!}{(s_n-l)!(s_n-l+1)!}n\sigma^{2(s_n-l)} (C_3s_n)^{2l}\left (\frac{n^\beta}{n^{2/\alpha- \epsilon}}\right )^{2l}.
\end{equation}
This is less than or equal to:
\begin{equation}\label{odd3A}
C_2\frac{(2s_n - 2l)!}{(s_n-l)!(s_n-l+1)!}n\sigma^{2s_n} \left(\frac{C_3s_n n^\beta}{n^{2/\alpha- \epsilon}}\right)^{2l}.
\end{equation}
The summation over $l$ of the above gives a contribution which is negligible with respect to $(2\sigma)^{2s_n}ns_n^{-3/2}$ since $s_n^2 \ll n^{\frac{2}{\alpha}-\epsilon - \beta}$.
   
Also, \cite{SosPec} gives us the following estimate for the contribution of odd paths coming from the class B:

\begin{equation}\label{odd1B}
\sum_{l=1}^{s_n-1}\sum_{J=1}^{2l} C^{2l} 4^J J! \binom{2l}{J} \binom {2s_n - 2l}{J}\frac{(2s_n - 2l)!}{(s_n-l)!(s_n-l+1)!}n\sigma^{2(s_n-l)}\left (\frac{n^\beta}{n^{2/\alpha- \epsilon}}\right )^{2l}.
\end{equation}

Again, proceeding as before there is a constant $K>0$ such that (\ref{odd1B}) divided by the contribution of even paths can be bounded by

\begin{equation}
\sum_{l=1}^{s_n-1} (K s_n n^{\beta - \frac{2}{\alpha}+\epsilon})^{2l},
\end{equation}
which tends to $0$ as $n$ goes to infinity. As shown in \cite{SosPec}, with a little bit of effort, the counting in case C can be reduced to the one in case A or B, which finishes the proof of the proposition.

\begin{rem}[Case $\alpha = 2$] When $\alpha = 2$, we do not use Proposition \ref{propw} as it is written to bound the largest eigenvalue of the truncated matrix. In fact, since one knows that, see \cite{Fel}, $$ f(x) = \E(a_{ij}^2\indi_{\{|a_{ij}|<x\}})$$ is a slowly varying function, we can show that for any $0<\delta <\epsilon$ $$\Pro \left(\lambda_{1}(\frac{A_1}{n^{\frac{2}{\alpha}-\epsilon+\delta}}) \geq 4\right) \rightarrow 0, \quad (n\rightarrow\infty).$$
 
\end{rem}

\end{proof}

\subsection{The largest eigenvalue of $A_2$}
In this subsection we will prove that the largest eigenvalue of $A_2$ is actually asymptotically given by its largest entry in absolute value. For short, we denote by $\hat{A}_{ij}$, $i, j = 1, \ldots, n$ the entries of the matrix $A_2$ and by $\hat{A}_{i_1j_1}$ the largest one in absolute value. The aim of this subsection is to prove:

\begin{prop}\label{prop1}
One has that for any $\epsilon > 0$ 
$$\Pro(|\frac{\lambda_{1}(A_2)}{\hat{A}_{i_1j_1}} - 1| > \epsilon) \rightarrow 0. \quad \text{as} \quad n \rightarrow \infty .$$
\end{prop}

Proposition \ref{prop1} will be enough to end the proof of the Corollary \ref{cor1} as we will explain now. First, we point that since the matrices that we are dealing are symmetric, $||A|| = \max\{|\lambda_{1}(A)|, |\lambda_{n}(A)|\}$ is exactly the operator norm of the matrix. Then triangular inequality implies:   

\begin{equation}\label{tri}
\lambda_{max}(\frac{1}{b_{n}}A_2) + \lambda_{min}(\frac{1}{b_{n}}A_1) \leq \lambda_{1}(\frac{1}{b_{n}}A_n) \leq \lambda_{max}(\frac{1}{b_{n}}A_1) + \lambda_{max}(\frac{1}{b_{n}}A_2).
\end{equation}

Now, since $\frac{n^{2/\alpha-\epsilon}}{b_{n}}$ goes to $0$ as $n$ goes to infinity, the boundedness in probability of the largest eigenvalue  of $\frac{1}{n^{2/\alpha-\epsilon}}A_1$ as proved in the last subsection implies that $\lambda_{max}(\frac{1}{b_{n}}A_1)$ goes to $0$ in probability. Then, by proposition \ref{prop1}, we have that the largest eigenvalue of $A_n$ behave just as the largest eigenvalue of $A_2$, i.e,    

\begin{equation}
\lim_{n\rightarrow \infty} \Pro(\lambda_{max}(\frac{1}{b_{n}}A_n)\leq x) = \lim_{n\rightarrow \infty} \Pro(\lambda_{max}(\frac{1}{b_{n}}A_2)\leq x) = \lim_{n\rightarrow \infty} \Pro(\frac{a_{i_1j_1}}{b_{n}}\leq x) = \exp(-x^{-\alpha}).
\end{equation} 

\begin{proof}[Proof of Proposition \ref{prop1}]
The proof of Proposition \ref{prop1} relies on the two following lemmas, repeating the arguments of Lemma \ref{lem1}. 

\begin{lem}\label{lem2}
Let $\epsilon > 0$ be fixed. With probability going to one, one has that
\begin{enumerate}
\item There are no diagonal entry $\hat{A_{ii}}$ greater in absolute value than $b_{n}^{1/2 + \epsilon}$.
\item There is no pair $(i,j)$ such that $|\hat{A}_{ij}| \geq b_{n}^{99/100}$ and $|\hat{A}_{ii}| + |\hat{A}_{jj}| \geq b_{n}^{1/10}$.

\item For any $\delta > 0$, there is no row that has two entries greater than $b_{n}^{3/4 + \delta}$.
\end{enumerate}
\end{lem}
\begin{proof}
(a) follows from basics results about extremes of $n$ independent variables. For (b),  one recalls (\ref{Luc}) to compute
\begin{eqnarray*}
\Pro(\exists (i,j)| |\hat{A}_{ij}| \geq b_{n}^{99/100} \text{and }  |\hat{A}_{ii}| + |\hat{A}_{jj}| \geq b_{n}^{1/10}) &\leq& 2\binom{n}{2} \bar{F}(b_{n}^{99/100})\bar{F}(b_{n}^{1/10})  \\
&\leq& n^{2} L(b_{n}^\frac{99}{100})L(b_{n}^\frac{1}{10})(b_{n}^\frac{-109\alpha}{100}) \\
&=& o(n^{-\frac{18}{100}+\theta})
\end{eqnarray*}
for some $\theta >0$ small enough. 

Similarly, another application of (\ref{Luc}) yields
\begin{eqnarray*}
\Pro(\exists i \leq n, \exists j, k \leq n, j\neq k, s.t.  |\hat{A}_{ij}| \geq b_{n}^{3/4 + \delta} \text{and } |\hat{A}_{ik}| \geq b_{n}^{3/4 + \delta}) &\leq& n^3 \bar{F}(b_{n}^{3/4 + \delta})\bar{F}(b_{n}^{3/4 + \delta})  \\
&\leq& n^{3} L(b_{n}^{3/4 + \delta})^2 \frac{1}{b_{n}^{3\alpha/2+2\delta \alpha}} \\
&=& o(n^{-4\delta+\theta})
\end{eqnarray*}
for some $\theta >0$ small enough. In the last line we have used that for any $\delta>0$, $b_{n}^{3\alpha/2+2\delta \alpha} \gg n^3$.
\end{proof}

We now show that the largest entry $\hat{A}_{i_1j_1}$ determines the largest eigenvalue of $A_2$. The idea is similar to the first part with minor changes. Introduce the vector 
$$ v_1 = \frac{1}{\sqrt{2}}(e_{i_1} \pm e_{j_1}),$$
where the sign $\pm$ is determined by the following rule: $\pm e_{j_1} = + e_{j_1}$ if $\hat{A}_{i_1j_1} \geq 0$ and $- e_{j_1}$ otherwise. Then, with probability going to one, $$\left\langle A_2v_1, v_1\right\rangle = \frac{1}{2}\hat{A}_{i_1i_1} + \frac{1}{2}\hat{A}_{j_1j_1} + |\hat{A}_{i_1j_1} | = |\hat{A}_{i_1j_1}|(1 + o(1)),$$ in view of item (b) of the preceding lemma and the fact that $\hat{A}_{j_1j_1}$ is the largest entry of the matrix $A_2$. 
Thus, again with probability going to one, 
$$ \lambda_{max}(A_2) \geq |\hat{A}_{i_1j_1}|(1 + o(1)).$$

We now turn to the upper bound which follows from the following Lemma.

\begin{lem}\label{lem3} One has that
$$\lim_{n\rightarrow\infty} \Pro(\exists i,  1 \leq i \leq n, \quad \max_{1 \leq j \leq n} |\hat{A}_{ij}| > b_{n}^{\frac{3}{4}+\frac{\alpha}{16}} \quad \text{and} \quad \sum_{j=1}^{n} |\hat{A}_{ij}| -\max_{1 \leq j \leq n} |\hat{A}_{ij}| >b_{n}^{\frac{3}{4}+\frac{\alpha}{16}})=0.$$
\end{lem}
\begin{proof}
Recall that by (\ref{beta}) we have $\alpha \beta > 1$.

Now, a trivial union bound gives us that
\begin{eqnarray}\label{Assa}
\Pro\left(\exists i, \sum_{j: n^{\beta}<|\hat{A}_{ij}|<b_{n}^{\frac{3}{4}+\frac{\alpha}{32}}}|\hat{A}_{ij}| \geq \frac{1}{2}b_{n}^{\frac{3}{4}+\frac{\alpha}{16}}\right) &\leq& \nonumber n \Pro\left(\#\{j:n^{\beta}<|\hat{A}_{ij}|<b_{n}^{\frac{3}{4}+\frac{\alpha}{32}}\} \geq \frac{b_{n}^\frac{\alpha}{32}}{2}\right) \\ \nonumber
&\leq& n^{1 + b_{n}^\frac{\alpha}{32}} \bar{F}(n^{\beta})^{b_{n}^\frac{\alpha}{32}}  \\ 
&\leq& n^{1 + b_{n}^{\frac{\alpha}{32}}(1-\alpha\beta)} \leq \exp(-n^\kappa)
 \end{eqnarray}
for some sufficiently small $\kappa > 0$. Thus, (\ref{Assa}) together with part (c) of Lemma \ref{lem2} yelds Lemma \ref{lem3}.

\end{proof}
 Now, combining Lemma \ref{lem3} with (c) of Lemma \ref{lem2} and repeating the same argument done in last section to prove part (b) in Lemma \ref{lem1}, Proposition {\ref{prop1}} holds.
 \end{proof}

\begin{proof}[Proof of Theorem \ref{thmw} when $2< \alpha < 4$]
By Proposition \ref{propw}, with probability going to $1$, we know that the spectrum of $\frac{1}{\sqrt{n^{2/\alpha - \epsilon}}}A_1$ is bounded. Also, let $A_{i_lj_l}$ be the $l-$th largest entry in absolute value of $A_2$. If one sets
\begin{equation*}
v_l=e_{i_l} \pm e_{j_l},
\end{equation*}
where $e_{i_l} \pm e_{j_l} = e_{i_l}+ e_{j_l}$ if $A_{i_lj_l} \geq 0$ and $e_{i_l}- e_{j_l}$ otherwise, then 
\begin{equation*}
A_2 v_l = \sum_{i=1}^n(A_{ii_l}\pm A_{ij_l})e_i = |A_{i_lj_l}|v_l + r,
\end{equation*}
where
\begin{equation*}
r = A_{i_li_l}e_{i_l}\pm A_{j_lj_l}e_{j_l}+\sum_{i=1,i\neq i_l,j_l}^{n}(A_{ii_l}\pm A_{ij_l})e_i.
\end{equation*}
By the same arguments that we used in Theorem \ref{thm1}, one can show that $\frac{\|r\|}{b_{n}}$ tends to $0$ as $n$ tends to infinity. This implies an existence of an eigenvalue $\lambda$ of $A_2$ such that $\lambda = |A_{i_lj_l}|(1+o(1))$. An application of Cauchy Interlacing Theorem just as before shows that in fact $\lambda = \lambda_l(A_2)$, where $\lambda_l(A_2)$ represents the $l-$th eigenvalue of $A_2$ in descending order.
Since by Weyl's inequalities, one has for all $l$ $, 1\leq l \leq n$,
\begin{equation}\label{Weyl}
\lambda_{l}(A_2)+\lambda_{n}(A_1) \leq \lambda_l(A_1+A_2) \leq \lambda_l(A_2)+\lambda_1(A_1) .
\end{equation}  

Dividing equation (\ref{Weyl}) by $b_n^2$ and taking the limit one gets 
\begin{equation*}
\lim_{n\rightarrow\infty} \Pro (\lambda_l(A_1+A_2) \leq b_{n} x) = \lim_{n\rightarrow\infty} \Pro (\lambda_l(A_2) \leq b_{n}x) = \lim_{n\rightarrow\infty} \Pro (a_{i_lj_l} \leq b_{n}x),
\end{equation*}
which ends the proof of (b).
\end{proof}


\section{Sample Covariance Matrices when $2 \leq \alpha < 4$.}\label{cov4}

This section heavily uses the previous ones. We repeat the arguments of section \ref{cov0} using the results of section \ref{wigner}.
As before, the first step is to truncate our matrix. 

Define
\begin{equation}
A_1=(\hat{A_{ij}}\indi_{|\hat{A_{ij}}|\leq n^{\beta}})_{i,j}, \quad A_2 = A_n - A_1,
\end{equation}
and also 
\begin{equation}
X_1= A_1A_1^t, \quad X_2 = X_n - X_1.
\end{equation}
It is clear that $ X_2 = A_1A_2^t + A_2A_1^t + A_2A_2^t$ and that the largest entries of $A_n$ belong to $A_2$. We are going to study the eigenvalues of $X_2$. The following lemma follows directly from the proof of Lemma \ref{lem1} and Lemma \ref{lem3}.

\begin{lem}\label{lem4}

Let $A_2$ be defined as above. Also assume that $p=\left\lfloor \gamma n\right\rfloor$ for some constant $\gamma \geq 1$.
Then:

\begin{enumerate}

	\item If $B_n^\delta$ is the event 'There is a row with $2$ entries greater than $b_{np}^\delta$ in absolute value' then 
	$$ \forall \delta > 3/4, \quad \lim_{p\rightarrow\infty} \Pro(B_n^\delta) = 0.$$
	
	\item One has that $$\lim_{p\rightarrow\infty} \Pro(\exists i,  1 \leq i \leq n, \quad \max_{1 \leq j \leq p} |a_{ij}| > b_{np}^{\frac{3}{4}+\frac{\alpha}{16}} \quad \text{and} \quad \sum_{j=1}^{p} |a_{ij}| - \max_{1 \leq j \leq p} |a_{ij}| >b_{np}^{\frac{3}{4}+\frac{\alpha}{16}})=0.$$
	
	\item One has that $$\lim_{p\rightarrow\infty} \Pro(\exists j,  1 \leq j \leq p, \quad \max_{1 \leq i \leq n} |a_{ij}| >b_{np}^{\frac{3}{4}+\frac{\alpha}{16}} \quad \text{and} \quad \sum_{j=1}^{p} |a_{ij}| - \max_{1 \leq j \leq p} |a_{ij}| >b_{np}^{\frac{3}{4}+\frac{\alpha}{16}})=0.$$

\end{enumerate}
\end{lem}

If one takes $v=(0,\ldots,0,1,0,\ldots,0)$ where the only non-zero entry of $v$ is in the position $i_1$, $\left\langle X_2v,v\right\rangle$ will give us the following bound:
\begin{equation}\label{akdof}
\lambda_{max}(X_2) \geq a_{i_1j_1}^2 (1+o(1)).
\end{equation} 

This inequality is justified by the preceding lemma, just as in theorem \ref{thm1}, and by the fact that the diagonal of $A_1A_2^t + A_2A_1^t$ has only zeros. More than that, as in section \ref{cov0}, we can also infer from Lemma \ref{lem4} that 
\begin{equation}\label{skpo}
\lambda_{max}(A_2A_2^t) = a_{i_1j_1}^2 (1+o(1)).
\end{equation} 

Now, using Rayleigh-Ritz representation and linearity of the scalar product, we have that
\begin{eqnarray*}
\lambda_{max} (X_n) &=& \max_{v: |v|=1} \langle X_nv, v \rangle = \max_{v: |v|=1} \langle (A_1 + A_2)(A_1 + A_2)^tv, v \rangle \nonumber \\
&=& \max_{v: |v|=1} \left(\langle X_1v, v \rangle +  \langle A_2A_2^tv, v \rangle + \langle A_1A_2^tv, v \rangle + \langle A_2A_1^tv, v \rangle \right)
\end{eqnarray*}
which, by Cauchy-Schwartz, yelds
\begin{eqnarray}\label{Ams2}
\lambda_{max} (X_n) &\leq& \max_{v: |v|=1}  \langle X_1v, v \rangle  + \max_{v: |v|=1}\langle A_2A_2^tv, v \rangle + 2 \max_{v: |v|=1}\|A_1^t v\|\max_{v: |v|=1}\|A_2^t v\| \nonumber \\ 
&\leq& \lambda_{max}(X_1) + \lambda_{max}(A_2A_2^t) + 2\left(\lambda_{max}(X_1)\lambda_{max}(A_2A_2^t)\right)^{1/2}.
\end{eqnarray}

In view of (\ref{akdof}), (\ref{skpo}), (\ref{Ams2}) and Weyl's inequality for $X_n = X_1 + X_2$, namely,
$$  \lambda_{max} (X_2) + \lambda_{min}(X_1) \leq \lambda_{max} (X_n),$$
it remains to show that the largest eigenvalue of $X_1$ is negligible with respect to $a_{i_1j_1}^2$ to conclude that 
$$ \lambda_{max} (X_n) = a_{i_1j_1}^2(1 + o(1))$$ and, therefore, finish the proof of Corollary  \ref{cor2}.

%
%
%

Thus, we turn back our attention to the truncated matrix $X_1 = A_1A_1^t$. As we did before, we will show that its largest eigenvalue properly normalized remains bounded on a set of probability arbitrarily close to $1$. Again, we study the asymptotics of some expectations of $X_1$.  

By the results of the last section, we just need to control the expected value of some traces of the matrix $X_1$.

Let $s_n$ be as in section \ref{comb}. One can write

\begin{equation}\label{tracecm}
\E(\text{Tr}X_1^{s_n})=\sum_\mathscr{P} \E a_{i_1i_0}a_{i_1i_2}a_{i_3i_2}a_{i_3i_4}\ldots a_{i_{2s_n-1}i_{2s_n-2}}a_{i_{2s_n-1}i_0}.
\end{equation}
where $\mathscr{P}$ denotes the set of all closed paths $P=\{i_0,i_1,\ldots,i_{2s_n-1},i_0\}$ with a distinguished origin, in the set $\{1,2,\ldots,p\}$ with the restriction $i_{t} \in \{1,2,\ldots,n\}$ for odd $t$. Now, since $a_{ij} \neq a_{ji}$ orientation of an edge plays a role.

We say that a path is odd if the number of passages in the direction $i \rightarrow j$ plus the number of passages in the direction $j \rightarrow i$ is odd for some $i$ and $j$. $P$ is even if it is not odd. Since we can center the random variables $a_{ij}\indi_{|a_{ij}|\leq n^{\beta}}$ just as in section \ref{wigner}, odd paths with a non-zero contribution have at least $3$ passages in one direction of an odd edge. We will prove that the contribution of odd paths to the sum (\ref{tracecm}) is negligible and that contribution of even paths can be easily bounded by the results on the Wigner case. To do so, one can proceed as follows. 

Construct a $p \times p$ random symmetric matrix $M = (y_{ij})_{1\leq i,j\leq p}$ such that $y_{ij}$ are independent identically distributed random variables with the same distribution as $a_{11}$. Hence, if we denote $\E(P)$ as the contribution of the path,

\begin{equation}\label{tracesum}
\sum_{P \in \mathscr{P}, P \text{even}} \E(P) \leq \E(\text{Tr}M^{2s_n}).
\end{equation}

In fact, one has a $1-1$ relation between paths in the LHS and paths that give a non-zero contribution to the sum in the RHS. Furthermore, if an edge $(i,j)$ is read $r$ times from left to right and $s$ times in the opposite direction, then $p + s = 2q$ for some integer $q$ and we can use the inequality 
\begin{equation}\label{lala}
\E a_{ij}^r\E a_{ji}^s \leq  \E |a_{ij}|^r\E |a_{ji}|^s \leq \E y_{ij}^{2q},
\end{equation} 
leading to (\ref{tracesum}).
 
Therefore, proceeding as in section \ref{wigner}, equation (\ref{tracesum}) implies:

\begin{prop}
There exists a constant $C>0$ such that $\E \left( \text{Tr} (\frac{X_1}{n^{4/\alpha-2\epsilon}})^{2s_n}) \right) < (2\sigma)^{4s_n}\frac{C n}{s_n^{3/2}}$. Moreover, 
$$\Pro(\lambda_{max}(X_1) \geq 8\sigma^2n^{4/\alpha-2\epsilon}) \leq \frac{\E \left( \text{Tr} (\frac{X_1}{n^{4/\alpha-2\epsilon}})^{2s_n}) \right)}{(8\sigma^2)^{2s_n}}\leq \exp(-\eta s_n),$$
for some constant $\eta > 0.$
\end{prop}
 
To prove that the contribution of odd paths is negligible with respect to $(2\sigma)^{4s_n}\frac{C n}{s_n^{3/2}}$, we can 
still bound the contribution of each path using the inequality (\ref{lala}). Since to an odd path $\mathcal{P}$, there corresponds an odd path in the expansion of  $\E(\text{Tr}M^{2s_n})$,
we analyze 
their contribution as in section $\ref{wigner}$.  
\begin{equation}\label{tracel}
\sum_{P \in \mathscr{P}, P \text{odd}} \E(P) \ll (2\sigma)^{4s_n}\frac{C n}{s_n^{3/2}}.
\end{equation}
 
This ends the proof Corollary \ref{cor2} in the case of $2 \leq \alpha<4$. The proof of Theorem \ref{thm1} in this case is identical to the proof of the case $0<\alpha<2$. 

\section{Acknowledgements}\nonumber
The authors would like to thank anonymous referees 1 and 2 for the careful and meticulous reading of the paper. Their comments and corrections helped in making this paper more coherent and error free. Research supported in part by NSF Grant OISE-0730136.

\end{document}